\documentclass[a4paper, reqno, 12pt]{amsart}

\usepackage[utf8]{inputenc}
\usepackage{amsthm,amsfonts,amssymb,amsmath,amsxtra,amsrefs}
\usepackage{mathtools}
 \usepackage{bbold}
\usepackage{latexsym}
 \usepackage{stmaryrd}
\usepackage{graphicx}
\usepackage{tikz}
\usepackage{tikz-cd}
\usepackage{todonotes,cancel}
\usepackage[margin=1.25in]{geometry}
\usepackage{mathrsfs}
\usepackage{bm}

\usepackage[all]{xy}
\SelectTips{cm}{}
\usepackage{xr-hyper}
\usepackage[colorlinks=false,
   citecolor=Black,
   linkcolor=Red,
   urlcolor=Blue]{hyperref}
\usepackage{verbatim}
\RequirePackage{xspace}
\RequirePackage{etoolbox}
\RequirePackage{varwidth}
\RequirePackage{enumitem}
\RequirePackage{tensor}
\RequirePackage{mathtools}
\RequirePackage{longtable}
\RequirePackage{multirow}

\usepackage{rotating}

\tikzset{
    labl/.style={anchor=south, rotate=90, inner sep=.5mm}
}

\newtheorem{thm}{Theorem}
\newtheorem{theorem}[thm]{Theorem}
\newtheorem{thmintro}{Theorem}

\newtheorem{proposition}[thm]{Proposition}

\newtheorem{lemma}[thm]{Lemma}

\newtheorem{corollary}[thm]{Corollary}

\theoremstyle{definition}
\newtheorem{definition}[thm]{Definition}

\newtheorem{remarks}[thm]{Remarks}

\numberwithin{equation}{section}
\numberwithin{thm}{section}

\newcommand{\qbinom}[2]{\genfrac{[}{]}{0pt}{0}{#1}{#2}}

\newcommand{\I}{I}
\newcommand{\coroot}{\alpha^\vee}
\newcommand{\A}{\mathcal{A}}
\newcommand{\U}{\mathrm{U}}

\newcommand{\AU}{{_\mathcal{A}\dot{\U}}}

\newcommand{\one}{\mathbf{1}}

\newcommand{\B}{\mathrm{B}}

\def\lr#1#2{\ensuremath{\left(\kern-.3em\left(\genfrac{}{}{0pt}{}{#1}{#2}\right)\kern-.3em\right)}}

\newcommand{\lie}{\text{Lie}}
\newcommand{\T}{\texttt T}
\newcommand{\Qq}{\mathbb{Q}(q)}

\newcommand{\WD}{\widehat{\Delta}}

\newcommand{\cc}{\mathbf{c}}

\newcommand{\G}{\mathbf{G}}

\newcommand{\BC}{\ensuremath{\mathbb {C}}\xspace}

\newcommand{{\BG}}{\ensuremath{\mathbb {G}}\xspace}

\newcommand{{\BK}}{\ensuremath{\mathbb {K}}\xspace}

\newcommand{\BN}{\ensuremath{\mathbb {N}}\xspace}

\newcommand{\BQ}{\ensuremath{\mathbb {Q}}\xspace}

\newcommand{\BZ}{\ensuremath{\mathbb {Z}}\xspace}

\newcommand{\RB}{\ensuremath{\mathrm {B}}\xspace}
\newcommand{\RO}{\ensuremath{\mathrm {O}}\xspace}

\begin{document}

\title[]{Symmetric subgroup schemes}

\author[Jinfeng Song]{Jinfeng Song}
\address{Department of Mathematics, The Hong Kong University of Science and Technology, Clear Water Bay, Hong Kong SAR.}
\email{jfsong@ust.hk}


\begin{abstract}
Chevalley group schemes are group schemes defined over the integers that parametrize connected reductive groups over algebraically closed fields as geometric fibers. In this paper, we construct closed subgroup schemes of Chevalley group schemes that parametrize symmetric subgroups of reductive groups as geometric fibers. Our construction relies crucially on the theory of quantum symmetric pairs and thus naturally admits a quantization. At the quantum level, this leads to the construction of coisotropic quantum right subgroups of the quantized function algebras of reductive groups.
\end{abstract}

\maketitle

\section{Introduction}

\subsection{Chevalley groups and quantum groups}
		
Let $G_{\mathbb{C}}$ be a connected reductive algebraic group over the complex numbers. We denote by $\RO_\BC$ the coordinate ring of $G_{\mathbb{C}}$.  In his famous papers \cites{Ch55, Ch61}, Chevalley constructed an integral form $\RO_\BZ$ of $\RO_\BC$ such that $\RO_\BC = \RO_\BZ \otimes_{\BZ} \BC$.   The integral form defines the \emph{Chevalley group scheme}  $\G_\BZ= Spec\, \RO_\mathbb{Z}$  over $\BZ$ such that the geometric fiber $G_k\footnote{We identify an algebraic variety over $k$ with its set of $k$-rational points.} = \G_{\BZ} \times_{{Spec\, \BZ}} {Spec\, } k$ is the connected reductive group associated with the given root datum for any algebraically closed field $k$.

Chevalley's approach depends on the choice of representations of $G_{\mathbb{C}}$. Kostant \cite{Ko66} identified (without proof) $\RO_\BZ$ intrinsically as the restricted dual of the Kostant's $\BZ$-form of the enveloping algebra of the Lie algebra of $G_{\mathbb{C}}$. Lusztig \cite{Lu09} reformulated (and proved) Kostant's construction using his theory of canonical bases on the quantum group $\U$. The coordinate ring $\RO_\BZ$ is identified as a dual subspace of the modified quantum group (at $q=1$) spanned by the dual canonical basis.

	\subsection{The symmetric subgroups} Let $G_k$ be a connected reductive algebraic group over an algebraically closed field $k$ of characteristic $\neq 2$. The fixed point subgroup $K_k=G_k^{\theta_k}$ is called a \emph{symmetric subgroup} of $G_k$. It was shown by Steinberg \cite{Ste68} that $K_k$ is reductive. We remark that $K_k$ may not be connected. Symmetric subgroups constitute significant examples of spherical subgroups and play a crucial role in recent advancements in the Relative Langlands Program \cites{BZSV,SV}. It is natural to extend the constructions of the Relative Langlands Program to the scheme-theoretic setting, making a scheme-theoretic construction of symmetric subgroups particularly important.
	
	To begin our construction, we need a unified classification of the symmetric subgroups. When $G_k$ is simple, it was shown by Springer \cite{Spr87} that involutions of $G_k$ are classified in terms of Satake diagrams which is independent of the defining field $k$, provided the characteristic of $k$ is not 2. We reformulate and generalize Springer's classification to all the  reductive groups using {\em $\imath$root data} in Section \ref{sec:fci}. 
		
	\subsection{Quantum symmetric pairs}
    Similar to Lusztig's construction of Chevalley group schemes, our approach crucially relies on quantization. Associated to the $\imath$root datum, we can construct a quantum symmetric pair $(\U,\U^\imath)$, where $\U^\imath \subset \U$ is a coideal subalgebra of the quantum group associated to the root datum of $G_k$. The quantum symmetric pair $(\U,\U^\imath)$ is a quantization of the pair of enveloping algebras of the symmetric pair $(\mathfrak{g}_k, \mathfrak{g}_k^{\theta_k})$. Here $\mathfrak{g}_k$ denotes the Lie algebra of $G_k$, and $\mathfrak{g}_k^{\theta_k}$ denotes the Lie algebra of $K_k$ (cf. \cite[\S 9.4]{Bo91}).  We often call $\U^\imath$ the {\em $\imath$quantum group}.
		
		Quantum symmetric pairs were originally introduced by Letzter \cite{Le99}, generalized by Kolb \cite{Ko14} to Kac-Moody cases.  The theory of canonical bases arising from quantum symmetric pairs was initiated by Bao and Wang \cite{BW18}. We refer to the survey \cite{Wang22} for recent developments in quantum symmetric pairs. 
  
  Let ${}_\A\dot{\U}^\imath$ be the $\A = \BZ[q,q^{-1}]$-form of the modified $\imath$quantum group. This is a free $\A$-subalgebra of the modified $\imath$quantum group $\dot{\U}^\imath$ with basis $\dot{\B}^\imath$. We call $\dot{\B}^\imath$ the $\imath$canonical basis of ${}_\A\dot{\U}^\imath$. 
		
		For any commutative ring $R$ and ring homomorphism $\A \rightarrow R$, we write ${}_R\dot{\U}^\imath=  R \otimes_{\A} {}_\A\dot{\U}^\imath$ if there is no ambiguity. We abuse notations and denote by  $\dot{\B}^\imath$ the basis of ${}_R\dot{\U}^\imath$ after base change.

	\subsection{Symmetric subgroup schemes} 
	
	We consider the ring homomorphism $\A \rightarrow \BZ$ mapping $q$ to $1$ and the algebra ${}_\BZ\dot{\U}^\imath=  \BZ \otimes_{\A} {}_\A\dot{\U}^\imath$.  Let $\RO^\imath_{\BZ}$ be the subspace of $\text{Hom}_{ {\BZ}}( {}_\BZ\dot{{\U}}^\imath,  \BZ)$ spanned by the dual basis of $\dot{\B}^\imath$. Let us state the main theorem of this paper. 
	
	\begin{thmintro}[Theorem~\ref{thm:Hopfi} \& Proposition~\ref{prop:GAi} \& Theorem~\ref{thm:Oik}]\label{thm:1}  \phantom{x}
	
	\begin{enumerate}
	\item  The subspace $\RO^\imath_{{\BZ}}$ is naturally a commutative Hopf algebra. 
    \item The algebra $\RO^\imath_{{\BZ}}$ defines a closed subgroup scheme of the Chevalley group scheme $\G$, denoted by $\G^\imath \subset \G$.  
    \item  Let $A$ be an integral domain with characteristic not 2. Then $\RO^\imath_{A} = \RO^\imath_{{\BZ}} \otimes_\BZ A$ is reduced.
    \item  Let $k$ be any algebraically closed field $k$ of characteristic $\neq 2$. We identify $\G^\imath_k$ with its set of $k$-rational points, denoted by $G^\imath_k$. 
	We have $G^\imath_k = K_k \subset G_k$. In particular, $\RO^\imath_{k}$ is the coordinate ring of the symmetric subgroup $K_k$.
	\end{enumerate}

	\end{thmintro}
  We call $\G^\imath$ the {\em symmetric subgroup scheme} over $\BZ$. 
   Theorem \ref{thm:1} generalizes Lusztig's construction \cite{Lu09} of the Chevalley group scheme $\G$ using (dual) canonical bases arising from quantum groups. 

    Let us mention that Lusztig's construction in \cite{Lu09} applies to the reductive group $K_k$ on its own. One can construct the corresponding Chevalley group scheme $\mathbf{K}$ over $\BZ$ using the relevant quantum group. However, $\mathbf{K}$ is usually not a subgroup scheme of $\G$ (over $\BZ$). This reflects the incompatibility (e.g., no natural embeddings) between the quantum groups associated to $K_k$ and $G_k$, respectively. From the point of view of Poisson-Lie groups, this corresponds to the fact that $K_\mathbb{C}\subset G_{\mathbb{C}}$ is only a coisotropic subgroup, and is not a Poisson subgroup in general.

    Our construction admits a natural quantization, that is, we have an $\A$-vector space $\RO_\A^\imath$ which is a quotient of the quantized coordinate algebra $\RO_\A$ of the reductive group. In Section \ref{sec:qt}, we show that $\RO_\A^\imath$ is a coisotropic quantum subgroup of $\RO_\A$ in the sense of Ciccoli \cite{Cic97}. 

\subsection{Historical remarks and applications}

In the work \cite[\S~3]{BS22} joint with Huanchen Bao, we have obtained the results in this paper for symmetric subgroups of \emph{quai-split} types. The main issue for other types was that the Strong Stability Theorem \cite{Wa23} (Theorem \ref{thm:stab}) for $\imath$canonical bases was only available for quasi-split cases back then. This theorem provides a crucial tool for establishing finiteness conditions related to $\imath$canonical bases, which form the foundation of our construction. See Section \ref{sec:fd} for details. The Strong Stability Theorem was later proved by Watanabe for all symmetric pairs. Therefore we are able to prove all the results in \cite[\S 3]{BS22} in this generality. 

The original motivation for the work \cite{BS22} was the study of the geometry of $K$-orbit closures in flag manifolds, where symmetric subgroup schemes were used to connect quantum algebras with geometry. We soon realized that the construction of symmetric subgroup schemes holds its own interest. In our later works \cites{BS24, BS25}, we employed symmetric subgroup schemes (for general types) to study symmetric spaces and their equivariant embeddings in positive characteristic. In the recent work of Watanabe \cite{Wa24b}, (quantized) symmetric subgroup schemes inspired the study of integrable modules over $\imath$quantum groups. For these reasons, we believe that the construction of symmetric subgroup schemes for general types merits careful documentation. This is the aim of this paper.

Some contents of the paper overlaps with the preprints \cite[\S~3]{BS22}. The main improvement is that we obtain all the results in \cite[\S~3]{BS22} for all symmetric pairs, which was previously established only for quasi-split types. Section \ref{sec:qt} is new.

\vspace{.2cm}

\noindent {\bf Acknowledgment: }The author is grateful to Huanchen Bao for his continuous support. The author is supported by the Glorious Sun Charity Fund.

\section{Preliminaries}

\subsection{Quantum groups}
We recall basic constructions regarding quantum groups mainly following \cite{Lu93}. 
	
\subsubsection{Root data}\label{sec:rod}
	A \emph{Cartan datum} (cf. \cite[\S 1.1.1]{Lu93}) (of finite type) is a pair $(\I,\cdot)$ consisting of a finite set $\I$ and a positive definite symmetric bilinear form on $\mathbb{Z}[\I]$ taking values in $\mathbb{Z}$ such that
    \begin{itemize}
        \item $i\cdot i\in\{2,4,6\}$ for any $i\in\I$;
        \item $2\frac{i\cdot j}{i\cdot i}\in\{0,-1,-2,-3\}$ for any $i\neq j$ in $\I$.
    \end{itemize}
The positive definite symmetric bilinear form defines a Euclidean space $\mathbb{R}^{|\I|}$ with the canonical basis $\I$. Then $\I\subset \mathbb{R}^{|\I|}$ forms a set of simple roots of a root system in $\mathbb{R}^{|\I|}$. Let $W=\langle s_i\mid i\in\I\rangle$ be the Weyl group associated with the root system. It is known that Cartan data classify split semisimple Lie algebras. To classify reductive groups, we need to expand the data to contain more information.

Given a Cartan datum $(\I,\cdot)$, a \emph{root datum} $(Y,X,(\alpha_i^\vee)_{i\in\I},(\alpha_i)_{i\in\I})$ of type $(\I,\cdot)$ consists of 
	\begin{itemize}
	    \item two finitely generated free abelian groups $X$, $Y$ equipped with a perfect pairing $\langle\,,\,\rangle:Y\times X\rightarrow \mathbb{Z}$;
	    \item elements $\alpha_i$ $(i\in\I)$ in $X$ and elements $\alpha_i^\vee$ $(i\in\I)$ in $Y$, such that $\langle \alpha_i^\vee,\alpha_j\rangle=2\frac{i\cdot j}{i\cdot i}$ for any $i,j\in\I$.
	\end{itemize}
	
The isomorphism between two root data is defined in an evident manner. Let $\epsilon_i=\frac{i\cdot i}{2}$ for any $i\in\I$, and $c_{ij}=\langle\alpha_i^\vee,\alpha_j\rangle$ for $i,j\in\I$. Let $X^+=\{\mu\in X\mid\langle\alpha_i^\vee,\mu\rangle\geq 0,\forall i\in\I\}$ be the set of \emph{dominant weights}. Since given a root datum, the underlying Cartan datum is uniquely determined, we will not specify the type of a root datum in what follows.

Let $k$ be an algebraically closed field. For any triple $(G_k,T_k,B_k)$, where $G_k$ is a connected reductive group defined over $k$, $T_k$ is a maximal torus of $G_k$, and $B_k$ is a Borel subgroup of $G_k$ containing $T_k$, one can associate a root datum in the following way. 
	
Let $X=\text{Hom}(T_k,k^\times)$, and $Y=\text{Hom}(k^\times,T_k)$ be the character lattice and the cocharacter lattice. Then there is a natural pairing between $Y$ and $X$. Let $(\alpha_i)_{i\in\I}\subset X$ (resp., $(\coroot_i)_{i\in\I}\subset Y$) be the simple roots (resp., simple coroots) associated to the Borel subgroup $B_k$, where $\I$ is a finite index set. Then $(Y,X,(\alpha_i^\vee)_{i\in\I},(\alpha_i)_{i\in\I}))$ is a root datum. For fixed $G_k$, by choosing different pair $(T_k,B_k)$, one gets isomorphic root data. We call any root datum in this isomorphism class the \emph{root datum associated to $G_k$}. There is a bijection between isomorphism classes of connected reductive groups over algebraically closed field, and isomorphism classes of root data.
    
	\subsubsection{Definitions} \label{sec:defq}
	We next recall the definition of quantum groups. Fix a root datum $(Y,X,(\alpha_i^\vee)_{i\in\I},(\alpha_i)_{i\in\I})$. Let $q$ be an indeterminate. We write $\A=\mathbb{Z}[q,q^{-1}]$ be the subring of $\mathbb{Q}(q)$. Write $q_i=q^{\epsilon_i}$ for all $i$ in $\I$. For $m,n,d\in \mathbb{Z}$ with $m\geq 0$, define 
	\begin{equation*}
	[n]=\frac{q^n-q^{-n}}{q-q^{-1}} \quad \text{ and } \quad [m]!=[1][2]\cdots[m].
	\end{equation*}
	These are called \emph{q-integers} and \emph{q-factorials}. Also define \emph{q-binomial coefficients}:
	\begin{equation*}
	\qbinom{n}{d}=\frac{[n][n-1]\cdots[n-d+1]}{[d]!},\text{ for $d\ge 0$}. 
	\end{equation*}
	Similarly define $[n]_i$, $[m]_i!$ and $\qbinom{n}{d}_i$ with $q$ replaced by $q_i$. Note that these are all elements in $\A$.
	
	The quantum group $\U$ associated to the given root datum is the associated $\mathbb{Q}(q)$-algebra (with 1) with generators 
	\begin{equation*}
	E_i \quad (i\in\I),\qquad F_i \quad (i\in\I), \qquad K_\mu \quad ( \mu\in Y),
	\end{equation*}
	which subjects to the relations
	\begin{align*}
	&K_0=1,\quad K_\mu K_{\mu'}=K_{\mu+\mu'}, \text{ for all }\mu,\,\mu'\in Y;\\
	&K_\mu E_i=q^{\langle \mu,\alpha_i\rangle}E_i K_\mu, \text{ for all }i\in\I,\, \mu\in Y;\\
	&K_\mu F_i=q^{-\langle \mu,\alpha_i\rangle}F_i K_\mu, \text{ for all }i\in\I,\, \mu\in Y;\\
	&E_iF_j-F_jE_i=\delta_{ij}\frac{K_i-K_{i}^{-1}}{q_i-q_i^{-1}}, \text{ for all }i,j\in\I, \text{ where }K_i=K_{\coroot_i}^{\epsilon_i} ;\\
	&\sum_{n=0}^{1-c_{ij}}(-1)^nE_i^{(n)}E_jE_i^{(1-c_{ij}-n)}=0,\qquad
	\sum_{n=0}^{1-c_{ij}}(-1)^nF_i^{(n)}F_jF_i^{(1-c_{ij}-n)}=0.
	\end{align*}
Here $E_i^{(n)}=E_i^n/[n]_i!$ and $F_i^{(n)}= F_i^n/[n]_i!$ are the divided powers.

There is a unique algebra incolution $\omega:\U\rightarrow\U$ such that $\omega(E_i)=F_i$, $\omega(F_i)=E_i$, $\omega(K_\mu)=K_{-\mu}$ for $i\in\I$, $\mu\in Y$. 
	
	\subsubsection{Canonical bases}
    
	 Let $\dot{\U}$ be the modified quantum group \cite[\S 23.1]{Lu93} and $\dot{\RB}$ be its canonical basis (\cite[\S 25.2.1]{Lu93}). Let $\AU$ be the $\A$-form of $\dot{\U}$, which is an $\A$-subalgebra of  $\dot{\U}$ generated by $E_i^{(n)} 1_\lambda$ and $F_i^{(n)} 1_\lambda$ for various $i \in I$, $n \ge 0$ and $\lambda \in X$. Then $\AU$ is a free $\A$-module with basis $\dot{\RB}$. We denote by $\AU^{>0}$ (resp., $\AU^{<0}$) the $\A$-subalgebra generated by $E_i^{(n)} 1_\lambda$ (resp., $F_i^{(n)} 1_\lambda$) for various $i \in I$, $n \ge 0$ and $\lambda \in X$.

     For any $\lambda\in X^+$, let $L(\lambda)$ be the (type I) simple $\U$-module (as well as a $\dot{\U}$-module) with highest weight $\lambda$. Let $\RB(\lambda)$ be the canonical basis of $L(\lambda)$. Let $v_\lambda \in L(\lambda) $ be the highest weight vector in $\RB(\lambda)$. Let $^\omega L(\lambda)$ to be the simple $\U$-module (as well as a $\dot{\U}$-module) where the action is twisted by $\omega$. It has the lowest weight $-\lambda$. We denote by ${}^\omega \RB(\lambda)$ the canonical basis of $^\omega L(\lambda)$ and denote by $v_{-\lambda}^-$ the unique canonical basis element of weight $-\lambda$. One has the following finiteness property \cite[\S 29.1.6]{Lu93}. 

     (a) {\emph{ For any $\U$-module $M$ which is isomorphic to a direct sum of simple modules, there are only finitely many canonical basis element $b\in\dot{\RB}$, such that $b$ acts on $M$ by a non-zero linear map.}}

     For $\lambda,\mu\in X^+$, let $\RB(\lambda,\mu)$ be the canonical basis of the module $^\omega L(\lambda)\otimes L(\mu)$ (\cite[Theorem 24.3.3]{Lu93}). Then one has the following compactibility property \cite[Theorem 25.2.1]{Lu93}.

     (b) {\emph{ The map $\dot{\U}\rightarrow {^\omega L(\lambda)\otimes L(\mu)}$, $u\mapsto u\cdot v_{-\lambda}^{-}\otimes v_\mu^+$, sends elements in $\dot{\RB}$ to either zero or elements in $\RB(\lambda,\mu)$, and the kernel is spanned by a subset of $\dot{\RB}$.}}

\subsubsection{Braid group actions}

 For $i\in\I$, $w\in W$, let $\T_i=\T_{i,+1}''$ and $\T_w=\T_{w,+1}''$ be Lusztig's braid group operators on $\U$ (\cite[\S 37.1]{Lu93}), which we recall as follows.

 \begin{align*}
    &\T_i(E_i)=-F_iK_i,\qquad \T_i(F_i)=-K_i^{-1}E_i;\\
    &\T_i(E_j)=\sum_{r=0}^{-c_{ij}}(-1)^rq_i^{-r}E_i^{(-c_{ij}-r)}E_jE_i^{(r)}\quad \text{for }j\neq i;\\
    & \T_i(F_j)=\sum_{r=0}^{-c_{ij}}(-1)^rq_i^{r}F_i^{(r)}F_jF_i^{(-c_{ij}-r)}\quad \text{for }j\neq i;\\
    & \T_i(K_\mu)=K_{s_i\mu}\quad \text{for $\mu\in Y$}.
\end{align*}

For any $w\in W$, let $w=s_{i_1}s_{i_2}\cdots s_{i_k}$ be a reduced expression. For any sequence $\cc=(c_1,c_2,\cdots,c_k)\in \BN^k$, define the element
\[
x_\cc=E_{i_1}^{(c_1)}\cdot \T_{i_1}(E_{i_2}^{(c_2)})\cdots \T_{i_1}\cdots\T_{i_{k-1}}(E_{i_k}^{(c_k)})\in {_\A\U^+}.
\]
It is known (cf. \cite[Proposition 40.2.1]{Lu93}) that the set $\{x_\mathbf{c}\mid \mathbf{c}\in\mathbb{N}^k\}$ forms an $\A$-basis of an $\A$-submodule $_\A\U^+(w)$, which is independent of the choice of the reduced expression of $w$. Set $_\A\U^+(w)_{>}$ to be the $\A$-submodule of $_\A\U^+(w)$ spanned by $\{x_\cc \mid \cc\in \BN^k-(0,\cdots,0)\}$.

\subsection{Chevalley group schemes}
Take a root datum $(Y,X,(\alpha_i)_{i\in\I},(\alpha_i^\vee)_{i\in\I})$ as in Section \ref{sec:rod}. In this section we recall Lusztig's construction of the Chevalley group scheme $\G$ in \cite{Lu09}.
\subsubsection{Lusztig's construction}\label{sec:nocp}

Let $A$ be an $\A$-algebra. Let ${}_A\dot{\U} = A \otimes_{\A} {}{_\A}\dot{\U}$ and ${}_A\dot{\RB}=\{1\otimes b\mid b\in\dot{\RB}\}$. The set ${}_A\dot{\RB}$ is called the canonical basis of $_A\dot{\U}$. If there is no confusion, we will write ${}_A\dot{\RB}=\dot{\RB}$, and use $b$ to denote the image $1\otimes b$.

Let $_A\widehat{\mathrm{U}}$ be the $A$-module consisting of all the formal linear combinations
\[
    \sum_{a\in\dot{\RB}} n_a a, \quad n_a\in A.
\]
One has a natural embedding $_A\dot{\U}\subset {}_A\widehat{\U}$. By \cite[\S 1.11]{Lu09}, $_A\widehat{\U}$ has a structure of $A$-algebra extending the algebra structure on $_A\dot{\U}$.
    
 Let $_A\widehat{\mathrm{U}}^{(2)}$ be the $A$-module consisting of all the formal linear combinations
\[
    \sum_{(a,a')\in \dot{\RB}\times \dot{\RB}} n_{a,a'}a\otimes a', \quad n_{a,a'}\in A.
\]
Then ${}_A\widehat{\U}^{(2)}$ also has an $A$-algebra structure by \cite[\S 1.11]{Lu09}. One has an $A$-algebra homomorphism $\WD:{_A\widehat{\mathrm{U}}}\longrightarrow {_A\widehat{\mathrm{U}}}^{(2)}$, compatible with the coproduct on $_A\dot{\U} $ by \cite[\S 1.17]{Lu09}.
Let $S:{_A\dot{\U}}\rightarrow {_A\dot{\U}}$ be the antipode map. It extends to $\widehat{S}:{_A\widehat{\U}}\rightarrow {_A\widehat{\U}}$. 

Let $\RO_A$ be the $A$-submodule of $_A\dot{\U}^*=\text{Hom}_A({_A\dot{\U}},A)$, which is spanned by $\{b^*\mid b\in \dot{\RB}\}$, where $b^* \in {}_A\dot{\U}^*$ is the \emph{dual canonical basis} element which sends $b$ to 1, and sends other canonical basis element to 0. Then for any element $f$ in $\RO_A$, it extends to an $A$-linear form on the completion $_A\widehat{\U}$. We write $\hat{f}$ to denote this extension.

Lusztig defined an $A$-Hopf algebra structure on $\RO_A$, which is essentially the dual of the Hopf algebra structure on $_A\widehat{\U}$ (\cite[\S 3.1]{Lu09}). Let $\delta:\RO_A\rightarrow\RO_A\otimes\RO_A$ be the coproduct, $\sigma:\RO_A\rightarrow\RO_A$ be the antipode, and $\epsilon:\RO_A\rightarrow A$ be the counit. The following commuting diagrams describe the product, coproduct and antipode structure on $\RO_A$.

For any $f$, $g \in \RO_A$, it follows from \cite[\S 3.1]{Lu09} that the following diagrams commute
\begin{equation*}
    \begin{tikzcd}
    & _A\widehat{\U} \arrow[r,"\WD"] \arrow[rd,"\widehat{fg}"'] & {_A\widehat{\U}}^{(2)} \arrow[d,"f\otimes g"] \\ &  & A
    \end{tikzcd}
    \begin{tikzcd}
    &  {_A\widehat{\U}}\otimes{_A\widehat{\U}} \arrow[r,"m"] \arrow[rd,"\delta(f)"'] & {_A\widehat{\U}} \arrow[d,"\hat{f}"] \\ & & A
    \end{tikzcd}
  \quad
    \begin{tikzcd}
    _A\widehat{\U} \arrow[r,"\widehat{S}"] \arrow[rd,"\widehat{\sigma(f)}"'] & {_A\widehat{\U}} \arrow[d,"\hat{f}"] \\ & A
    \end{tikzcd}.
\end{equation*}
Here $f\otimes g: {_A\widehat{\U}^{(2)}}\rightarrow A$ sends formal linear combination $\sum_{(a,a')\in \dot{\RB}\times\dot{\RB}}n_{a,a'}a\otimes a'$ to $\sum_{(a,a')\in \dot{\RB}\times\dot{\RB}}n_{a,a'}f(a) g(a')$, where the second sum is finite, thanks to the definition of $\RO_A$. Similarly $\delta(f)\in \RO_A\otimes\RO_A$ induces a well-defined $A$-linear form on $_A\widehat{\U}\otimes{_A\widehat{\U}}$. We still use $\delta(f)$ to denote this map.

For the rest of this section, we assume that $A$ is a commutative ring where $q$ acts by 1 as an $\A$-module. By \cite[Theorem 3.15]{Lu09}, $\RO_A$ is a commutative Hopf algebra over $A$. The set $\text{Hom}_{A-\text{alg}}(\RO_A, A)$, consisting of $A$-algebra homomorphisms $\RO_A\rightarrow A$ preserving 1, has a group structure.

By definition, $\text{Hom}_{A-\text{alg}}(\RO_A,A)$ is an $A$-submodule of $\RO_A^{*}$, the $A$-linear duals of $\RO_A$. There is an $A$-linear bijective map  from $\RO_A^{*}$ to $_A\widehat{\U}$, given by $\phi\mapsto \sum_{b\in\dot{\RB}}\phi(b^*)b$. Under this bijection, the image of the subset $\text{Hom}_{A-\text{alg}}(\RO_A,A)$ is 
\begin{equation}\label{eq:ga}
G_A=\{\xi=\sum_{b\in\dot{\RB}}n_b b\in {_A\widehat{\mathrm{U}}}\mid \WD(\xi)=\xi\otimes\xi,\,n_{1_0}=1\}.
\end{equation}
Here for any $\xi=\sum_{b\in \dot{\RB}}n_bb$ and $\xi'=\sum_{b'\in\dot{\RB}}n'_{b'}b'$ in $_A\widehat{\U}$, the element $\xi\otimes\xi'$ in $_A\widehat{\U}^{(2)}$ is defined to be the formal linear combination
\[
\xi\otimes\xi'=\sum_{(b,b')\in \dot{\RB}\times \dot{\RB}} n_{b}n'_{b'}b\otimes b'.
\]

It is clear that the subset $G_A$ is closed under taking products. It moreover admits a group structure, where the unit is $\sum_{\lambda\in X}\one_\lambda$, and the inverse is given by the restriction of the antipode $\widehat{S}:{_A\widehat{\U}}\rightarrow{_A\widehat{\U}}$. The bijection $\text{Hom}_{A-\text{alg}}(\RO_A,A)\xrightarrow{\sim}G_A$ is a group isomorphism.

\begin{definition}
Define the $\BZ$-group scheme $\G$ by setting $\G(A)=G_A$ \footnote{Here we identify a $\BZ$-group scheme with the associated $\BZ$-group functor following \cite{Jan03}} for any (unital) commutative ring $A$. Then $\G\cong Spec\,\RO_\BZ$ is called the Chevalley group scheme associated to the given root datum.    
\end{definition}

\subsubsection{The reductive group $G_k$ and its Lie algebra}\label{sec:regpq}
Let $k$ be an algebraically closed field. Set $G_k$ to be the subset of $_k\widehat{\U}$ defined as above by letting $A=k$. By \cite[Theorem~4.11]{Lu09}, $G_k$ is a connected reductive group over $k$ with the coordinate ring $\RO_k$. Let $T_k$ be the subset of $G_k$ consisting of elements of the form $\sum_{\lambda\in X}n_\lambda\one_\lambda$, such that $n_\lambda n_{\lambda'}=n_{\lambda+\lambda'}$. Let $_k\widehat{\U}^{>0}$ be the $k$-subspace of $_k\widehat{\U}$ consisting of all elements of the form $  \sum_{b\in \dot{\RB} \cap {}_k\dot{\U}^{>0},\lambda\in X}n_b(b\one_\lambda)$ with $n_b\in k$. Set $G_k^{>0}=G_k\cap {_k\widehat{\U}^{>0}}$ and $B_k=G_k^{>0}T_k$. It follows from \cite[Theorem 4.11]{Lu09} that $T_k$ is a maximal torus of $G_k$ and $B_k$  is a Borel subgroup of $G_k$.

We have the following isomorphisms as free abelian groups
\begin{align*}
X&\xrightarrow{\sim} \text{Hom}(T_k,k^\times),\qquad &Y\xrightarrow{\sim} &\text{Hom}(k^\times,T_k)\\
\lambda & \longmapsto \big(\sum_{\lambda\in X}n_\lambda\one_\lambda\mapsto n_\lambda\big),\qquad &\gamma\longmapsto &\big(a\mapsto \sum_{\lambda\in X}a^{\langle \gamma,\lambda\rangle}\one_\lambda\big).
\end{align*}
These two maps give an isomorphism between the root data defining the quantum group with the root data associated to $(G_k,T_k,B_k)$.

For any linear algebraic group $H_k$ over $k$, we use the notation $\lie(H_k)$ to denote the Lie algebra, which is a subalgebra of $k[H_k]^*$, the linear dual of the coordinate ring of $H_k$. 

Recall that $\RO_k$ is identified with the coordinate ring of the connected reductive group $G_k$. We next describe its Lie algebra as a subalgebra of $_k\widehat{\U}$. 

We firstly identify $\RO_k^*$ with $_k\widehat{\U}$, via $\mu\mapsto\sum_{b\in\dot{\RB}}\mu(b^*)b$. For any $i\in\I$ and $\mu\in Y$, define the following elements in $_k\widehat{\U}$:
\[
e_i=\sum_{\lambda\in X}E_i\one_\lambda,\qquad f_i=\sum_{\lambda\in X}F_i\one_\lambda,\qquad h_\mu=\sum_{\lambda\in X} \langle \mu,\lambda \rangle \one_\lambda.
\]
Here we use the same notations to denote the elements in the algebra after base change.
 
Then the Lie algebra $\lie(G_k)$ is identified as the Lie subalgebra of $_k\widehat{\U}$ generated by $e_i$ ($i\in\I$), $f_i$ ($i\in\I$), and $h_\mu$ ($\mu\in Y$).


\subsection{The $\imath$quantum groups}

In this section, we recall the definitions and important properties of $\imath$quantum groups.
\subsubsection{The $\imath$root data}\label{sec:irt}

Given a Cartan datum $(\I,\cdot)$, a \emph{Satake diagram} $(\I=\I_\bullet\sqcup\I_\circ,\tau)$ consists of a bipartition $\I=\I_\bullet\sqcup\I_\circ$ and an involution $\tau:\I\rightarrow\I$ satisfying certain conditions (cf. \cite{Ar62}). Let $W_{\I_\bullet}\subset W$ be the parabolic subgroup generated by simple reflections $s_i$ with $i\in\I_\bullet$, and let $w_\bullet$ be the longest elements in $W_{\I_\bullet}$. We shall omit the precise definition for the Satake diagram but mention the following two properties: (1) the map $\tau$ preserves the bilinear form on $\mathbb{Z}[\I]$ and leaves $\I_\circ$ and $\I_\bullet$ invariant; (2) for any $i\in\I_\bullet$ one has $\tau i=-w_\bullet i$.  We refer the readers to \cite[Table 4]{BW18} for a list of irreducible Satake diagrams. It is known that Satake diagrams classify symmetric pairs of semisimple Lie algebras. In order to classify symmetric pairs of reductive groups, we shall expand the data and consider the \emph{$\imath$root data}.

\begin{definition}\label{def:iroot}
    Given a Satake diagram $(\I=\I_\bullet\sqcup\I_\circ,\tau)$, an \emph{$\imath$root datum} of type $(\I=\I_\bullet\sqcup\I_\circ,\tau)$ is a tuple $(Y,X,\{\alpha_i^\vee\}_{i\in\I},\{\alpha_i\}_{i\in\I},\theta)$ consisting of
    \begin{itemize}
    \item a root datum $(Y,X,\{\alpha_i^\vee\}_{i\in\I},\{\alpha_i\}_{i\in\I})$;
        \item an involution $\theta$ on $X$ as abelian groups, such that $\theta(\alpha_i)=-w_\bullet\alpha_{\tau i}$ for any $i\in\I$. 
    \end{itemize}
\end{definition}

Given an $\imath$root datum, its underlying Satake diagram is uniquely determined. Therefore, we will not specify its type when referring to an $\imath$root datum. The involution $\theta$ on $X$ induces an involution on $Y$ via the perfect pairing. When there is no confusion, we still use $\theta$ to denote this involution on $Y$.

Following \cite[(3.3)]{BW18}, we define the \emph{$\imath$-weight lattice} and the \emph{$\imath$-root lattice} as follows
\[
X_\imath=X/\Breve{X},  \text{ where }\Breve{X}=\{\lambda-\theta(\lambda)\mid \lambda\in X\},\quad \text{ and } \quad Y^\imath=\{\mu\in Y\mid \theta(\mu)=\mu\}.
\]

For any $\lambda\in X$, we denote by $\overline{\lambda}$ the image of $\lambda$ in $X_\imath$. There is a well-defined pairing $Y^\imath\times X_\imath\rightarrow\BZ$, $(\mu,\overline{\lambda}) \mapsto \langle \mu,\lambda\rangle$, for  $\mu\in Y^\imath, \lambda\in X$. We denote pairing by $\langle \cdot,\cdot\rangle$ again. Note that this pairing is not perfect in general.

\begin{lemma}\label{le:nop}
The abelian group $X_\imath$ has no odd torsion. Namely, for any $\overline{\lambda} \in X_\imath$ and any odd $n\in\BN$ with $n\overline{\lambda}=0$, we have  $\overline{\lambda}=0$.
\end{lemma}

\begin{proof}
Let  $\BZ[2^{-1}] \subset \BQ$ be the localization of $\BZ$ inverting $2$. By definition, we have the short exact sequence
\[
0\rightarrow \Breve{X}\rightarrow X\rightarrow X_\imath\rightarrow 0.
\]
Tensoring with the flat $\BZ$-module $\BZ[2^{-1}]$, we get
\[
0\rightarrow \Breve{X}\otimes \BZ[2^{-1}]\rightarrow X\otimes \BZ[2^{-1}]\rightarrow X_\imath\otimes \BZ[2^{-1}]\rightarrow 0.
\]

Note that the $\BZ[2^{-1}]$-module homomorphism 
\[
X\otimes\BZ[2^{-1}]\longrightarrow \Breve{X}\otimes \BZ[2^{-1}]
\]
sending $\lambda\otimes 1$ to $(\lambda-\theta\lambda)\otimes 1/2$ gives a splitting for the above short exact sequence. Hence $X_\imath\otimes \BZ[2^{-1}]$ can be viewed as a submodule of $X\otimes \BZ[2^{-1}]$, which is torsion-free. We deduce that $X_\imath\otimes \BZ[2^{-1}]$ is torsion-free. Then by the structure theory of finitely generated abelian groups, we see that $X_\imath$ has no odd torsion. 
\end{proof}

\subsubsection{Definitions}\label{sec:par}

We define the \emph{$\imath$quantum group} $\U^\imath$ to be the $\Qq$-subalgebra of $\U$ generated by the following elements (\cite[Definition~3.5]{BW18})
\begin{align*}
 B_i=F_i+\varsigma_i\T_{w_\bullet}(E_{\tau i})K_i^{-1}\; (i\in\I_\circ), \quad 
K_\mu \; (\mu\in Y^\imath),\quad F_i\; (i\in\I_\bullet),\quad E_i\; (i\in\I_\bullet).
\end{align*}
Here $\varsigma_i\in \pm q^{\mathbb{Z}}$ are parameters satisfying certain conditions in \cite[Definition~3.5]{BW18}. We will only recall these conditions later when they are needed. Let $\dot{\U}^\imath$ be the modified algebra of $\U^\imath$ and let $\dot{\RB}^\imath$ be the canonical basis of $\dot{\U}^\imath$ defined in \cite[\S3.7, \S6.4]{BW18}. The algebra $\dot{\U}$ is naturally a $(\dot{\U}^\imath, \dot{\U}^\imath)$-bimodule by \cite[\S3.7]{BW18}.
Let $_\A\dot{\U}^\imath$ be the $\A$-form of $\dot{\U}^\imath$ \cite[Definition~3.19]{BW18}, which is an $\A$-subalgebra of $\dot{\U}^\imath$ as well as a free $\A$-module with basis $\dot{\RB}^\imath$. 

For any $\A$-algebra $A$, set ${}_A\dot{\U}^\imath=A\otimes_{\A} {_\A\dot{\U}^\imath}$. Then the following claim holds.

(a) {\emph{ For any $\A$-algebra $A$, we have an $A$-algebra embedding ${}_A\dot{\U}^\imath \rightarrow  {_A\widehat{\mathrm{U}}}$, $x \mapsto \sum_{\lambda \in X} x \one_\lambda$. Here $x \one_\lambda \in {_\A\dot{\U}}$ is via the ${}_A\dot{\U}^\imath$-action on ${}_A\dot{\U}$.}}

	Since  ${}_\A\dot{\U}^\imath$ is a free $\A$-module by  \cite[Theorem~6.17]{BW18}, it suffices to prove the claim for $A = \A$, which follows from \cite[\S3.7]{BW18} and \cite[\S1.11]{Lu09}.

\subsubsection{The strong stability theorem}

For $\lambda,\mu\in X^+$, let $L^\imath(\lambda,\mu)=\U\cdot v^+_{w_\bullet\lambda}\otimes v^+_\mu$ be the $\U$-submodule of $L(\lambda)\otimes L(\mu)$ generated by the element $v^+_{w_\bullet\lambda}\otimes v^+_\mu$. By \cite[Proposition 6.8]{BW18}, there is a canonical basis $\RB^\imath(\lambda,\mu)$ of the module $L^\imath(\lambda,\mu)$. The following theorem was conjectured by Bao--Wang in \cite[Remark 6.18]{BW18} under the name of the strong stability conjecture. It was proved by Watanabe \cite[Theorem 4.3.1]{Wa23}.

\begin{theorem}(\cite[Theorem 4.3.1]{Wa23})\label{thm:stab}
    Under the proper choice of the parameters $\varsigma_i$ $(i\in\I_\circ)$, the map $\dot{\U}^\imath\rightarrow L^\imath(\lambda,\mu)$, $u\mapsto u\cdot v_{w_\bullet \mu}^+\otimes v_\lambda^+$, sends elements in $\dot{\RB}^\imath$ to either zero or elements in $\RB^\imath(\lambda,\mu)$, and the kernel is spanned by a subset of $\dot{\RB}^\imath$.
\end{theorem}

One explicit choice of the parameters $\{\varsigma_i\mid i\in\I_\circ\}$ such that Theorem \ref{thm:stab} holds was given in \cite[Section 4]{Wa23a}. For the rest of the paper, we assume that the parameters are chosen in such a way.  

\subsubsection{A projection map}\label{sec:proj}

Recall that $w_\circ$ is the longest element in $W$, and $w_\bullet$ is the longest element in the parabolic subgroup $W_{I_\bullet}.$ Set $w^\bullet=w_\circ w_\bullet$. For $\lambda\in X$, we write $_\A\dot{\U}(w^\bullet)_>\one_\lambda={_\A\dot{\U}}{_\A\U^+(w^\bullet)_{>}}\one_\lambda$. 

We denote by $\iota_\lambda: {_\A\dot{\U}^\imath\one_{\overline{\lambda}}} \rightarrow {_\A\dot{\U}\one_\lambda}$ the map $ x \mapsto x \one_\lambda$. We define the composition
\begin{equation}\label{eq:pim}
p_{\imath,\lambda}=\pi_\lambda\circ\iota_\lambda:{_\A\dot{\U}^\imath\one_{\overline{\lambda}}}\xrightarrow{\iota_\lambda}{_\A\dot{\U}\one_\lambda}\xrightarrow{\pi_\lambda}{_\A\dot{\U}\one_\lambda}/_\A\dot{\U}(w^\bullet)_>\one_\lambda.
\end{equation}
Here $\pi_\lambda$ is the canonical projection map. It follows from \cite[Corollary 6.20]{BW18} that $p_{\imath,\lambda}$ is an isomorphism of $\A$-modules.

Let $A$ be any $\A$-algebra.  Since $_\A\dot{\U}(w^\bullet)_>\one_\lambda$ is a free $\A$-submodule of $_\A\dot{\U}\one_\lambda$ by \cite[Proposition~40.2.1]{Lu93}, we deduce that $A\otimes \big({_\A\dot{\U}\one_\lambda}/_\A\dot{\U}(w^\bullet)_>\one_\lambda\big)\cong {_A\dot{\U}\one_\lambda}/_A\dot{\U}(w^\bullet)_>\one_\lambda$ as $A$-modules.  Here $_A\dot{\U}(w^\bullet)_>\one_\lambda={_A\dot{\U}}{_A\U^+(w^\bullet)_{>}}\one_\lambda$. We deduce that the composition
\begin{equation}\label{eq:pa}
p_{\imath,\lambda}=\pi_\lambda\circ\iota_\lambda:{_A\dot{\U}^\imath\one_{\overline{\lambda}}}\xrightarrow{\iota_\lambda}{_A\dot{\U}\one_\lambda}\xrightarrow{\pi_\lambda}{_A\dot{\U}\one_\lambda}/_A\dot{\U}(w^\bullet)_>\one_\lambda
\end{equation}
is an isomorphism as $A$-modules. Here we abuse the notation by using the same notations to denote the maps after the base change. 

\subsection{Symmetric subgroups}

We recall the symmetric subgroups and their classifications following Springer \cite{Spr87}.

\subsubsection{Definitions}\label{sec:lkb}

Let $k$ be an algebraically closed field with characteristic not 2. Let $G_k$ be a connected reductive algebraic group defined over $k$. A \emph{symmetric pair} $(G_k,\theta_k)$ consists of a connected reductive group $G_k$, and an involution $\theta_k$ of $G_k$. Two symmetric pairs $(G_k,\theta_k)$ and $(G'_k,\theta_k')$ are called isomorphic if there exists an isomorphism between algebraic groups $f:G_k\rightarrow G'_k$, such that $f\circ\theta_k=\theta_k'\circ f$.

A torus $S_k$ of $G_k$ is called \emph{spit} if $\theta_k(s)=s^{-1}$ for any $s\in S_k$. Let $B_k$ be a Borel subgroup of $G_k$ such that $\theta_k(B_k)\cap B_k$ has the minimal dimension among all the choices of Borel subgroups. Take a maximal torus $T_k\subset B_k$, which contains a maximal split torus. Such a pair $(B_k,T_k)$ is called a \emph{split pair}. The existence of split pairs was proved in \cite[\S 1.3]{Spr87}.

Let us associate an $\imath$root datum to a symmetric pair. Let $(Y,X,(\alpha_i^\vee)_{i\in\I},(\alpha_i)_{i\in\I}))$ be the root datum associated to $(G_k,T_k,B_k)$. Since $\theta_k$ preserves $T_k$, it induces an involution on $X$. Following \cite{Spr87}, there is a Satake diagram $(\I=\I_\bullet\sqcup\I_\circ,\tau)$, such that $\theta_X(\alpha_i)=w_\bullet \alpha_{\tau i}$, for any $i\in\I$. Then $(Y,X,\{\alpha_i^\vee\}_{i\in\I},\{\alpha_i\}_{i\in\I},\theta)$ is an $\imath$root datum (Definition \ref{def:iroot}), which is called the $\imath$root datum associated with $(G_k,\theta_k)$. Note that different choices of the split pair $(B_k,T_k)$ give the isomorphic $\imath$root data. 

\subsubsection{Classifications}\label{sec:class}

When $G_k$ is simple, Springer \cite{Spr87} showed that each Satake diagram corresponds to a symmetric pair $(G_k,\theta_k)$, and hence obtained a bijection between the isomorphism classes of symmetric pairs (of simple groups) with the Satake diagrams. His classification used the classification of order two elements in simple groups. 

In Section \ref{sec:fci}, by giving functorial constructions of the symmetric subgroups associated with any $\imath$root data, we generalize Springer's classification to the symmetric pairs of any reductive groups in terms of the $\imath$root data.

\section{Symmetric subgroup schemes}

\subsection{The finiteness conditions}\label{sec:fd}

Recall  the canonical basis $\dot{\RB}$ of $\dot{\U}$ and the $\imath$canonical basis $\dot{\RB}^\imath$ of $\dot{\U}^\imath$. Recall the $\imath$weight lattice in Section \ref{sec:irt}. 

For $\lambda\in X$, recall the isomorphism of $\A$-modules from (\ref{eq:pim}): 
\[
p_{\imath,\lambda}=\pi_\lambda\circ\iota_\lambda:{_\A\dot{\U}^\imath\one_{\overline{\lambda}}}\xrightarrow{\iota_\lambda}{_\A\dot{\U}\one_\lambda}\xrightarrow{\pi_\lambda}{_\A\dot{\U}\one_\lambda}/_\A\dot{\U}(w^\bullet)_>\one_\lambda.
\]
 Set $s_\lambda=p_{\imath,\lambda}^{-1}\circ\pi_\lambda:{_\A\dot{\U}\one_\lambda}\longrightarrow{_\A\dot{\U}^\imath}\one_{\overline{\lambda}}.$ Then $s_\lambda\circ \iota_\lambda=id$. 

Let $$\dot{\RB}^\imath_{\overline{\lambda}} =\dot{\RB}^\imath\cap{\dot{\U}^\imath\one_{\overline{\lambda}}}\qquad\text{
 and }\qquad  \dot{\RB}_\lambda=\dot{\RB}\cap{\dot{\U}\one_\lambda}.$$ 
Then ${\dot{\U}^\imath\one_{\overline{\lambda}}}$ has basis $\dot{\RB}^\imath_{\overline{\lambda}}$ and ${\dot{\U}\one_\lambda}$ has basis $\dot{\RB}_\lambda$ (cf. \cite[Theorem 25.2.1]{Lu93} \cite[Theorem 6.17]{BW18}). 

We write  $$\iota_\lambda(b)=\sum_{b'\in\dot{\RB}_\lambda}\iota_{\lambda,b,b'}b', \quad \text{for } b\in\dot{\RB}^\imath_{\overline{\lambda}},  \iota_{\lambda,b,b'} \in \A; $$
$$s_\lambda(b)=\sum_{b'\in\dot{\RB}_{\overline{\lambda}}}s_{\lambda,b,b'}b',\quad \text{for }  b\in\dot{\RB}_\lambda,  s_{\lambda,b,b'} \in \A.$$
We view $_\A\dot{\U}^\imath$ as an $\A$-subalgebra of $_\A\widehat{\U}$ via the embedding  $x \mapsto \sum_{\lambda \in X}\iota_\lambda(x) $ as in \S \ref{sec:par} (a). We write 
\[
    b=\sum_{b'\in\dot{\RB}}\iota_{b,b'}b' \in {}_\A\widehat{\U}, \quad \text{for } b\in\dot{\RB}^\imath,  \iota_{b,b'} \in \A.
\]

\begin{lemma}\label{le:finit}
Let $\mu\in X$. Then for any $b'\in \dot{\RB}_\mu$, there are only finitely many $b \in \dot{\RB}^\imath_{\overline{\mu}}$, such that $\iota_{\mu,b,b'}\neq 0$. For any $b'\in\dot{\RB}^\imath_{\overline{\mu}}$, there are only finitely many $b\in\dot{\RB}_\mu$, such that $s_{\mu,b,b'}\neq 0$. In particular,  for any $b'\in \dot{\RB}$, there are only finitely many $b\in\dot{\RB}^\imath$, such that $\iota_{b,b'}\neq 0$.
\end{lemma}

\begin{proof}

Suppose there exists some $b'\in\dot{\RB}_\mu$, such that there are infinitely many $b\in\dot{\RB}^\imath_{\overline{\mu}}$, such that $\iota_{\mu,b,b'}\neq 0$. 

Take $\mu_1, \mu_2\in X^+$, such that $b'\cdot v_{-\mu_1}^-\otimes v_{\mu_2}^+\neq 0$ in $^\omega L(\mu_1)\otimes L(\mu_2)$. Then by the assumption, we get infinitely many $b\in\dot{\RB}^\imath_{\overline{\mu}}$, such that $b\cdot v_{-\mu_1}\otimes v_{\mu_2}\neq 0$.

By \cite[Theorem 4.18]{BW18} we have $\U^\imath$-module isomorphism $$\mathcal{T}:{ L(\tau\mu_1)}\rightarrow {^\omega L(\mu_1)}$$ sending $v_{\tau\mu_1}^{\bullet}$ to $v_{-\mu_1}^-$. Here $v_{\tau\mu_1}^{\bullet}\in{L(\tau\mu_1)}$ is the unique canonical basis element of weight $w_\bullet(\tau\mu_1)$. Then we have the $\U^\imath$-module isomorphism $$\mathcal{T}^{-1}\otimes id:{^\omega L(\mu_1)}\otimes L(\mu_2)\rightarrow L(\tau\mu_1)\otimes L(\mu_2)$$ sending $v_{-\mu_1}^{-}\otimes v_{\mu_2}^+$ to $v_{\tau \mu_1}^{\bullet}\otimes v_{\mu_2}^+$. Hence via this isomorphism, we get infinitely many $\imath$canonical basis elements $b$, with $b\cdot v_{\tau \mu_1}^{\bullet}\otimes v_{\mu_2}^+\neq 0$. This contradicts Theorem \ref{thm:stab}. We proved the first claim.

For the second claim, suppose there exists some $b'\in \dot{\RB}^\imath_{\overline{\mu}}$, such that there are infinitely many $b\in\dot{\RB}_\mu$, with $s_{\mu,b,b'}\neq 0$. We take $\mu_1,\mu_2\in X^+$, such that $b'\cdot v^+_{w_\bullet\mu_1}\otimes v^+_{\mu_2}\neq 0$ in $L(\mu_1)\otimes L(\mu_2)$. Note that for any $x\in {_\A\dot{\U}}\one_\mu$, we have $x\cdot v_{w_\bullet\mu_1}^+\otimes v_{\mu_2}^+=s_\mu(x)\cdot v_{w_\bullet\mu_1}^+\otimes v_{\mu_2}^+$. Hence for any $b\in\dot{\RB}_\mu$, we have $b\cdot v_{w_\bullet\mu_1}^+\otimes v_{\mu_2}^+\neq 0$ whenever $s_{\mu,b,b'}\neq0$. Therefore we get infinitely many canonical basis element $b$, with $b\mid_{L(\mu_1)\otimes L(\mu_2)}\not\equiv 0$, which is not possible by \cite[29.1.6 (c)]{Lu93}. This proves the second claim.
\end{proof}

\begin{corollary}\label{cor:fimu}
For any $b',b''\in\dot{\RB}^\imath$, let us write $b'\cdot b''=\sum_{b\in\dot{\RB}^\imath}m_{b',b''}^b b$. Then for any $b\in\dot{\RB}^\imath$, there are only finitely many pairs $(b',b'')$ in $\dot{\RB}^\imath\times \dot{\RB}^\imath$, such that $m_{b',b''}^b\neq 0$.
\end{corollary}

\begin{proof}
For any $b\in\dot{\RB}^\imath$, we choose $\mu_1,\mu_2\in X^+$, such that $b\cdot v^+_{w_\bullet\mu_1}\otimes v^+_{\mu_2}\neq 0$ in $L(\mu_1)\otimes L(\mu_2)$. Suppose $m_{b',b''}^b\neq 0$, for some $b',b''$ in $\dot{\RB}^\imath$. Then thanks to Theorem \ref{thm:stab}, we have $(b'\cdot b'')\cdot v^+_{w_\bullet\mu_1}\otimes v^+_{\mu_2}\neq 0$. In particular, $b'\mid_{L(\mu_1)\otimes L(\mu_2)}\not\equiv 0$, and $b''\mid_{L(\mu_1)\otimes L(\mu_2)}\not\equiv 0$. Finally, by \cite[29.1.6 (c)]{Lu93}, there are only finitely many canonical basis element $c\in\dot{\RB}$, such that $c\mid_{L(\mu_1)\otimes L(\mu_2)}\not\equiv 0$. Then thanks to Lemma \ref{le:finit}, we only have finitely many such pairs $(b',b'')$ in $\dot{\RB}^\imath\times\dot{\RB}^\imath$.
\end{proof}

\subsection{The definition}\label{sec:def}
Let $A$ be an $\A$-algebra. We follow the same notations and conventions as in Section \ref{sec:nocp}. We write $_A\dot{\U}^\imath=A\otimes_\A\big({_\A\dot{\U}^\imath}\big)$. We will abuse the notations, and use $\dot{\RB}^\imath$ to denote the basis of $_A\dot{\U}^\imath$, consisting of the image of the $\imath$canonical basis elements after base change. This should not cause any confusion.

Let $_A\widehat{\mathrm{U}}^\imath$ be the $A$-module consisting of formal linear combinations
\[ \sum_{b\in\dot{\RB}^\imath} n_b b, \quad n_b\in A.
\]
Thanks to Lemma \ref{le:finit}, $_A\widehat{\mathrm{U}}^\imath$ can be naturally embedded into $_A\widehat{\mathrm{U}}$. By Corollary \ref{cor:fimu}, the product structure on $_A\widehat{\U}^\imath$ is well defined, and the embedding $_A\widehat{\U}^\imath\hookrightarrow{_A\widehat{\U}}$ is compatible with products.

Define $_A\widehat{\U}^{\imath,1}$ be the $A$-module consisting of all the formal linear combinations
\[
\sum_{(a,a')\in \dot{\RB}^\imath\times \dot{\RB}} n_{a,a'}a\otimes a', \quad n_{a,a'}\in A.
\]
 Again, by Lemma \ref{le:finit}, we have natural embedding $_A\widehat{\U}^{\imath,1}\hookrightarrow{_A\widehat{\U}^{(2)}}$. Thanks to the Corollary \ref{cor:fimu} and \cite[Lemma 1.8]{Lu09}, one can define a product structure on $_A\widehat{\U}^{\imath,1}$ in an evident way. Then the embedding is moreover compatible with products.

Since $\U^\imath$ is a right coideal subalgebra of $\U$, it follows that $\WD:{_A\widehat{\U}}\rightarrow{_A\widehat{\U}}^{(2)}$ restricts to an $A$-algebra homomorphism
\begin{equation*} 
\WD: {_A\dot{\U}^\imath}\longrightarrow {_A\widehat{\U}^{\imath,1}}.
\end{equation*}

Define $\RO_A^\imath$ be the $A$-submodule of $(_A\dot{\U}^\imath)^*=\text{Hom}_A({_A\dot{\U}^\imath},A)$, spanned by the \emph{dual $\imath$canonical basis} $\{b^*\mid b\in \dot{\RB}^\imath\}$, where $b^*$ stands for the $A$-linear maps sending $b'$ to $\delta_{b', b}$, for any $b'\in\dot{\RB}^\imath$.

Recall from Section\ref{sec:nocp} that $\RO_A$ is an $A$-submodule of $_A\dot{\U}^*$, spanned by the dual canonical basis. Define the $A$-linear map 
\[
r:\RO_A\rightarrow\RO_A^\imath, \quad b^* \mapsto \sum_{b_1\in\dot{\RB}^\imath}\iota_{b_1,b}b_1^*, \quad \text{for }b\in\dot{\RB}.
\]
 The summation is finite thanks to Lemma \ref{le:finit}. Then it is clear that for any $f\in \RO_A$, we have $r(f)=\hat{f}\mid_{_A\dot{\U}^\imath}$. Here $\hat{f}$ is the extension of $f$ to $_A\widehat{\U}$.

For any $\mu\in X$, let $\RO_{A,\mu}$ be the $A$-submodule of $\RO_A$, spanned by all $b^*$, such that $b\in \dot{\RB}\cap {_A\dot{\U}\one_\mu}$. Then $\RO_A=\bigoplus_{\mu\in X}\RO_{A,\mu}$. Moreover, we have $\RO_{A,\mu}\cdot\RO_{A,\mu'}\subset \RO_{A,\mu+\mu'}$. Via the restriction, we have a natural map $\RO_A\rightarrow(_A\dot{\U}\one_\mu)^*$. Under this restriction, the submodule $\RO_{A,\mu}$ is sent injectively to an $A$-submodule of $(_A\dot{\U}\one_\mu)^*$. We identify $\RO_{A,\mu}$ with its image in $(_A\dot{\U}\one_\mu)^*$.

Similarly, for any $\zeta\in X_\imath$, define $\RO_{A,\zeta}^\imath$ be the $A$-submodule of $\RO_A^\imath$, spanned by all the elements $b^*$, such that $b\in \dot{\RB}^\imath\cap {_A\dot{\U}^\imath\one_\zeta}$. Then $\RO_A^\imath=\bigoplus_{\zeta\in X_\imath}\RO_{A,\zeta}^\imath$. Similarly, the submodule $\RO_{A,\zeta}^\imath$ can be identified with an $A$-submodule of $(_A\dot{\U}^\imath\one_\zeta)^*$ via restriction.

By Section \ref{sec:fd} we have $A$-linear maps
\begin{equation*}
\begin{tikzcd}
{_A\dot{\U}\one_\mu}\arrow[r,bend right,"s_\mu"']& {_A\dot{\U}^\imath\one_{\overline{\mu}}} \arrow[l,"\iota_\mu"'], 
\end{tikzcd}, \quad \text{ such that }s_\mu\circ\iota_\mu=id, \quad \text{for any } \mu\in X.
\end{equation*}

Taking the linear dual of these maps, and using Lemma \ref{le:finit}, we have
\begin{equation*}
\begin{tikzcd}
\RO_{A,\mu}\arrow[r,"\iota_\mu^*"] & \RO_{A,\overline{\mu}}^\imath \arrow[l,bend left,"s_\mu^*"]
\end{tikzcd}, \quad \text{such that }\iota_\mu^*\circ s_\mu^*=id.
\end{equation*}
In particular, we deduce that $\iota_\mu^*$ is surjective, and $s_\mu^*$ is injective.

It follows from the definition that $\iota_\mu^*(f)=r(f)$, for any $f\in\RO_{A,\mu}$. Hence we deduce 

(a) {\em the map $r:\RO_A\rightarrow\RO_A^\imath$ is surjective.}

For the rest of this section, we assume that $A$ is commutative and $q$ acts by 1.

Recall from Section \ref{sec:nocp} that $(\RO_A,\delta, \sigma,\epsilon)$ is a commutative $A$-Hopf algebra.



\begin{theorem}\label{thm:Hopfi}
The $A$-module $\RO_A^\imath$ has a structure of a commutative $A$-Hopf algebra, such that the surjection $r:\RO_A\rightarrow\RO_A^\imath$ is a Hopf algebra homomorphism. 
\end{theorem}

\begin{proof}
Note that the kernel $I_A$ of $r$ consists of linear forms $f$, such that $\hat{f}\mid_{{{}_A\dot{\U}^\imath}} = 0$. We firstly show that $I_A$ is an ideal. Take any $f\in I_A$, and $g\in \RO_A$. For any $x\in{_A\dot{\U}^\imath}$, we have $\widehat{fg}(x)=f\otimes g\circ(\widehat{\Delta}(x))=0$. The last equality follows from the fact $\widehat{\Delta}({_A\dot{\U}^\imath})\subset {_A\widehat{\U}^{\imath,1}}$. Hence $fg$ belongs to $I_A$. Since $\RO_A$ is commutative, $I_A$ is a two-sided ideal.

Since the embedding $_A\dot{\U}^\imath\hookrightarrow{_A\widehat{\U}}$ is compatible with products, it follows that the comultiplication $\delta:\RO_A\rightarrow\RO_A\otimes\RO_A$ induces $\RO_A/I_A\rightarrow \RO_A/I_A\otimes\RO_A/I_A$. Finally, it is direct to check that $\sigma(I_A)\subset I_A$, and $\epsilon(I_A)=0$. Hence $\RO_A/I_A$ admits a commutative $A$-Hopf algebra structure. We then endow $\RO_A^\imath$ with the commutative $A$-Hopf algebra structure via the isomorphism $\RO_A^\imath\cong \RO_A/I_A$ and the theorem is proved.
\end{proof}

Since $\RO_A^\imath$ is a commutative Hopf algebra over $A$, the set $\text{Hom}_{A-\text{alg}}(\RO_A^\imath, A)$ of $A$-algebra homomorphisms $\RO_A^\imath\rightarrow A$ preserving 1, has a group structure.

By definition, $\text{Hom}_{A-\text{alg}}(\RO_A^\imath,A)$ is an $A$-submodule of $\RO_A^{\imath,*}$, the $A$-linear duals of $\RO_A^\imath$. There is an ($A$-linear) bijective map  from $\RO_A^{\imath,*}$ to $_A\widehat{\U}^\imath$, given by $\phi\mapsto \sum_{b\in\dot{\RB}^\imath}\phi(b^*)b$. We write $G_A^\imath$ to be the image of $\text{Hom}_{A-\text{alg}}(\RO_A^\imath,A)$ under this bijection. Via the embedding $_A\widehat{\U}^\imath\hookrightarrow{_A\widehat{\U}}$, we view $G_A^\imath$ as a subset of ${_A\widehat{\U}}$. 

Recall the construction for the group $G_A$ from Section \ref{sec:regpq}. The following claim is immediate. 

(b) {\em As subsets of $_A\widehat{\U}$, we have $G_A^\imath={_A\widehat{\U}^\imath}\cap G_A$, which is a subgroup of $G_A$. And the bijection $\text{Hom}_{A-\text{alg}}(\RO_A^\imath,A)\xrightarrow{\sim} G_A^\imath$ is moreover a group isomorphism. } 

\begin{definition}
We define $\G^\imath$ to be the $\BZ$-group scheme, by setting $\G^\imath(A)=G_A^\imath$, for any $\BZ$-algebra A. We have $\G^\imath \cong Spec\, \RO_{\BZ}^\imath  $ as affine group schemes.
\end{definition}

By the claim (a), $\G^\imath$ is a closed subgroup scheme of $\G$, where $\G$ denotes the Chevalley group scheme over $\BZ$ associated to (the root data) of $G$.



We next show that $\G^\imath_{\BZ[2^{-1}]}\rightarrow Spec\,\BZ[2^{-1}]$ has reduced geometric fibres.
\begin{proposition}\label{prop:GAi}
Suppose $A$ is an integral domain with characteristic not 2. Then $\RO_A^\imath$ is a reduced $A$-algebra.
\end{proposition}

\begin{proof}
For any $\mu,\mu'\in X$, we have the following commutative diagram by definitions:
\begin{equation}\label{diag:au}
    \begin{tikzcd}
    {_A\dot{\U}^\imath}\one_{\overline{\mu+\mu'}} \arrow[r,"\iota_{\mu+\mu'}"] \arrow[d,"\Delta_{\overline{\mu},\overline{\mu'}}"'] & {_A\dot{\U}\one_{\mu+\mu'}} \arrow[d,"\Delta_{\mu,\mu'}"']  \\ {_A\dot{\U}^\imath\one_{\overline{\mu}}}\otimes{}{_A\dot{\U}^\imath\one_{\overline{\mu'}}} \arrow[r,"\iota_\mu\otimes\iota_{\mu'}"] & {_A\dot{\U}\one_\mu}\otimes{} {_A\dot{\U}\one_{\mu'}}. 
    \end{tikzcd}
\end{equation}

Here $\Delta_{\mu,\mu'}$ (resp., $\Delta_{\overline{\mu},\overline{\mu'}}$) stands for the comultiplication restricting to the corresponding weight spaces (resp., $\imath$weight spaces). 

For any $\lambda\in X$, it follows from the definition that $_A\dot{\U}(w^\bullet)_>\one_\lambda$ is the kernel of the map $s_\lambda:{}_A\dot{\U}\one_\lambda\rightarrow{}_A\dot{\U}^\imath\one_{\overline{\lambda}}$. 

For $\mu,\mu'\in X$, it follows from \cite[2.8 (a)]{Lu09} that 
\[
\Delta_{\mu,\mu'}\left({}_A\dot{\U}(w^\bullet)_>\one_{\mu+\mu'}\right)\subset \left({}_A\dot{\U}(w^\bullet)_>\one_{\mu}\right)\otimes{}_A\dot{\U}\one_{\mu'}+{}_A\dot{\U}\one_\mu\otimes\left({}_A\dot{\U}(w^\bullet)_>\one_{\mu'}\right).
\]

Together with \eqref{diag:au} and the equality $s_\lambda\circ\iota_\lambda=id$ for $\lambda\in X$, we deduce that the diagram
\begin{equation}\label{dia:As}
    \begin{tikzcd}
    {_A\dot{\U}\one_{\mu+\mu'}} \arrow[r,"s_{\mu+\mu'}"] \arrow[d,"\Delta_{\mu,\mu'}"'] & {_A\dot{\U}^\imath\one_{\overline{\mu+\mu'}}} \arrow[d,"\Delta_{\overline{\mu},\overline{\mu'}}"]\\ {_A\dot{\U}\one_\mu}\otimes{}{_A\dot{\U}\one_{\mu'}} \arrow[r,"s_\mu\otimes s_{\mu'}"] & {_A\dot{\U}^\imath\one_{\overline{\mu}}}\otimes {_A\dot{\U}^\imath\one_{\overline{\mu'}}}
    \end{tikzcd}
\end{equation}
commutes.

By taking the linear dual of the diagram (\ref{dia:As}) and restricting to the proper subspaces, we obtain the commutative diagram:
\begin{equation}\label{dia:OmA}
    \begin{tikzcd}
 \RO_{A,\mu+\mu'} & \RO_{A,\overline{\mu+\mu'}}^\imath \arrow[l,"s_{\mu+\mu'}^*"'] \\ \RO_{A,\mu}\otimes \RO_{A,\mu'} \arrow[u] & \RO_{A,\overline{\mu}^\imath}\otimes\RO_{A,\overline{\mu'}}^\imath \arrow[u] \arrow[l,"s_\mu^*\otimes s_{\mu'}^*"'].
 \end{tikzcd}
\end{equation}

Since $\RO_A^\imath$ is a flat $A$-module, it can be viewed as a subring of $\RO_K^\imath$, where $K$ is the fractional field of $A$. Hence we may assume that $A$ is a field. Set $p=\text{char }A$. Then by our assumption, $p\neq 2$.

Firstly assume $p>0$. Note that for any $\zeta,\zeta'\in X_\imath$, we have $\RO_{A,\zeta}^\imath\cdot\RO_{A,\zeta'}^\imath\subset\RO_{A,\zeta+\zeta'}^\imath$. So $\RO_A^\imath=\bigoplus_{\zeta\in X_\imath}\RO_{A,\zeta}^\imath$ is a $X_\imath$-graded algebra. We firstly prove that a homogeneous element cannot be nilpotent.

Take $f\in \RO_{A,\zeta}^\imath$, $g\in \RO_{A,\zeta'}^\imath$, such that $f\cdot g=0$. Choose $\mu,\mu'\in X$, such that $\overline{\mu}=\zeta$, $\overline{\mu'}=\zeta'$. By the commuting diagram (\ref{dia:OmA}), we deduce that $s_\mu^*(f)\cdot s_{\mu'}^*(g)=0$. Since $\RO_A$ is a integral domain, and $s_\mu^*$, $s_{\mu'}^*$ are injective, we deduce that either $f$ or $g$ is 0. Now suppose $f^n=0$, for some homogeneous element $f$, using the argument inductively, we deduce that $f=0$.

Suppose $f=\sum_{\zeta\in X_\imath} f_\zeta\in \RO_A^\imath$ be a nilpotent element, where $f_\zeta\in\RO_{A,\zeta}^\imath$. Then there exists $r\in \BN$, such that $f^{p^r}=0$. Note that $f^{p^r}=\sum_{\zeta\in X_\imath} f_\zeta^{p^r}$. Element $f_\zeta^{p^r}$ has degree $p^r\zeta$. Since $X_\imath$ has no $p$-torsion by Lemma \ref{le:nop}, we deduce that $f_\zeta^{p^r}$ and $f_{\zeta'}^{p^r}$ have different degrees whenever $\zeta\neq \zeta'$. Hence $f_\zeta^{p^r}=0$, for any $\zeta\in X_\imath$, which implies $f_\zeta=0$, thanks to the above argument. Therefore $f=0$.

The case when $p=0$ follows from a well-known result by Cartier (which asserts that any finitely generated Hopf algebra over a field with characteristic 0 is reduced).
\end{proof}

\begin{remarks}
The assumption that $\textrm{char}(A)\neq 2$ can not be dropped in general. In the case of of rank 1, one can show that $\RO_A^\imath\cong A[u,v]/(u^2-v^2-1)$, which is non-reduced if $A$ has characteristic 2.
\end{remarks}

\subsection{Functorial construction of involutions}\label{sec:fci}

Our next goal is to describe the points of the group scheme $\mathbf{G}^\imath$. We firstly construct an involution on the Chevalley group scheme $\mathbf{G}$.

Let $A$ be a commutative ring, which is viewed as an $\A$-algebra as before. Recall the group $G_A$ in (\ref{eq:ga}). Let us define a group involution $\theta_A$ of $G_A$. The $A$-algebra $_A\dot{\U}$ is naturally a $\BZ[\I]$-graded algebra, where $\text{deg}(E_i\one_\lambda)=i$, $\text{deg}(F_i\one_\lambda)=-i$, and $\text{deg}(\one_\lambda)=0$, for any $i\in\I$ and $\lambda\in X$. 

For any group morphism $t:\BZ[\I]\rightarrow A^\times=A-\{0\}$, we define an $A$-algebra automorphism $\Xi (t):{_A\dot{\U}}\longrightarrow {_A\dot{\U}}$, given by $\Xi(t)(u)=t(\mu)u$, for any homogeneous element $u$ in ${_A\dot{\U}}$ with degree $\mu\in\BZ[\I]$. Since the canonical basis elements are homogeneous, automorphism $\Xi(t)$ extends to the completion algebra $_A\widehat{\U}$.

Recall the parameters $\varsigma_i$ ($i\in\I_\circ$) from Section \ref{sec:par}. Let $\overline{\varsigma_i}\in A$ be the image of $\varsigma_i$ in $A$. Then we have $\overline{\varsigma_i}\in\{\pm 1\}$. By the assumption in \cite[Definition 3.5]{BW18}, one has $\overline{\varsigma_i}=\overline{\varsigma_{\tau i}}$ if $\langle \alpha_i^\vee,\theta \alpha_i\rangle=0$, and $\overline{\varsigma_i}\cdot\overline{\varsigma_{\tau i}}=(-1)^{\langle 2\rho_\bullet^\vee,\alpha_i\rangle}$, for any $i\in \I_\circ$. Here $2\rho_\bullet^\vee\in Y$ denotes the sum of positive coroots in the sub-root system generated by $\alpha_i$ for $i\in\I_\bullet$.

Let $\epsilon=\epsilon((\overline{\varsigma_i})_{i\in\I_\circ}):\BZ[\I]\rightarrow A^\times$ be the group homomorphism such that $\epsilon(i)=\overline{\varsigma_i}$, for any $i\in\I_\circ$; and $\epsilon(i)=-1$, for any $i\in \I_\bullet$. 

Recall the involution $\omega$ on $\U$ from Section \ref{sec:defq}. It induces an involution on $\dot{\U}$ in a natural way. It was proved in \cite{Lu23} that $\omega$ preserves the canonical basis $\dot{\RB}$. Therefore it extends to an involution $\omega:{}_A\dot{\U}\rightarrow{}_A\dot{\U}$. 

Recall the graph involution $\tau: \I\rightarrow\I$. It extends to an involution $\tau:X\rightarrow X$ by $\tau(\lambda) =-w_\bullet\theta\lambda$. By the definition of the $\imath$root datum we have $\tau(\alpha_i)=\alpha_{\tau i}$ for $i\in\I$. We define the algebra involution $\Tilde{\tau}:{}_\A\dot{\U}\rightarrow{}_\A\dot{\U}$ by $\Tilde{\tau}(E_i^{(n)}\one_\lambda)=E_{\tau i}^{(n)}\one_{\tau \lambda}$, $\Tilde{\tau}(F_i^{(n)}\one_\lambda)=F_{\tau i}^{(n)}\one_{\tau \lambda}$, for $i\in\I$, $\lambda\in X$, and $n\in\mathbb{N}$. Since all the construction regarding the canonical basis is equivariant with respect to the graph involution $\tau$, it is clear that $\Tilde{\tau}$ preserves $\dot{\RB}$. Therefore it entends to an involution $\Tilde{\tau}:{}_A\widehat{\U}\rightarrow{}_A\widehat{\U}$.

For any $w\in W$, the braid group action $\T_w$ on $\U$ induces an action on $_\A\dot{\U}$ (\cite[\S 41.1.2]{Lu93}). We firstly show the following finiteness condition.

\begin{lemma}\label{le:fib}
    For $w\in W$ and $b\in\dot{\RB}$, let us write $\T_w(b)=\sum_{b'\in\dot{\RB}}c^w_{b,b'}b'$, where $c^w_{b,b'}$ belongs to $\A$. Then for any fixed $b'\in\dot{\RB}$, there are only finitely many $b\in\dot{\RB}$, such that $c^w_{b,b'}\neq 0$.
\end{lemma}

\begin{proof}
    For any $w\in W$ and $b'\in\dot{\RB}$, let us take $\lambda,\mu\in X^+$, such that $b'\cdot v_{-\lambda}^-\otimes v_\mu^+\neq 0$ in $^\omega L(\lambda)\otimes L(\mu)$. By abuse of language, let $\T_w:{}^\omega L(\lambda)\otimes L(\mu)\rightarrow{} ^\omega L(\lambda)\otimes L(\mu)$ be the braid group action on the module. Suppose $c^w_{b,b'}\neq 0$ for $b\in\dot{\RB}$. Then 
    \[
    b\cdot \T_w^{-1}(v_{-\lambda}^-\otimes v_\mu^+)=\T_{w}^{-1}(\T_w(b)\cdot v_{-\lambda}^-\otimes v_\mu^+)\neq 0.
    \]
     Hence there are only finitely many $b\in\dot{\RB}$ satisfying this condition.
\end{proof}

Thanks to this lemma, $\T_w$ extends to the completion $_\A\widehat{\U}$. Therefore it induces an action on $_A\widehat{\U}$, which we still denote by $\T_w$.

Define
\[
\theta_A:=\T_{w_\bullet}\circ\tilde{\tau}\circ\omega\circ\Xi(\epsilon):{_A\widehat{\U}}\longrightarrow{_A\widehat{\U}}
\]
to be an automorphism of $_A\widehat{\U}$.

\begin{thm}\label{prop:ainv}
We have $\theta_A^2=id$, and $\theta_A(G_A)=G_A$. Therefore we have a group involution $\theta_A:G_A\rightarrow G_A$. Moreover this involution is functorial on $A$, that is, for any ring homomorphism $A\rightarrow B$ with the induced map $G_A\rightarrow G_B$, the following diagram
\begin{equation}
    \begin{tikzcd}
        & G_A \arrow[r] \arrow[d,"\theta_A"'] & G_B \arrow[d,"\theta_B"] \\ & G_A \arrow[r] & G_B.
    \end{tikzcd}
\end{equation}
commutes.
\end{thm}

\begin{proof}
When there is no confusion, let us write $\theta=\theta_A$ for simplicity. To show that $\theta^2=id$, it suffices to show that $\theta^2$ fixes canonical basis elements of $_A\dot{\U}$ pointwise. Since $_A\dot{\U}=A\otimes \big(_\BZ\dot{\U}\big)$, we only need to  check this over $\BZ$. Moreover, since $_\BZ\dot{\U}$ can be viewed as a subspace of $_\BC\dot{\U}$, it will suffice to check over $\BC$. Write $\dot{U}$ to denote $_\BC\dot{\U}$ for brevity.

The $\BC$-algebra $\dot{U}$ is generated by elements $\one_\lambda$ ($\lambda\in X$), $E_i\one_\lambda$ ($i\in\I$, $\lambda\in X$), $F_i\one_\lambda$ ($i\in\I$, $\lambda\in X$). It suffices to check these generators are fixed by $\theta^2$.

It is easy to see that $\theta(\one_\lambda)=\one_{-w_\bullet\tau\lambda}$ for any $\lambda\in X$. Since $-w_\bullet\tau$ is an involution on $X$, we deduce that $\theta^2(\one_\lambda)=\one_\lambda$.

Write $\omega'=\omega\circ\Xi(-1)$, and $\theta'=\T_{w_\bullet}\circ\tilde{\tau}\circ\omega'$. Then $\theta=\theta'\circ\Xi(-\epsilon)$. Here $-1:\BZ[\I]\rightarrow A^\times$ denotes the group homomorphism sending all the $i\in\I$ to $-1$. By \cite[\S 2.3(c)]{Lu09}, for any $i\in\I$, we have $\T_i(u)=\overline{s_i}u\overline{s_i}^{-1}$, for any $u\in {_A\widehat{\U}}$, where $\overline{s_i}=x_i(1)y_i(-1)x_i(1)$. Note that $\omega'(x_i(1))=y_i(-1)$, and $\omega'(y_i(-1))=x_i(1)$. Then by doing computation in $\textrm{SL}_2$, we have $\omega'(\overline{s_i})=y_i(-1)x_i(1)y_i(-1)=\overline{s_i}$. Hence for any $u\in{_A\widehat{\U}}$, we have $\omega'\circ\T_i(u)=\omega'(\overline{s_i}u\overline{s_i}^{-1})=\overline{s_i}\omega'(u)\overline{s_i}^{-1}=\T_i\circ\omega'(u)$. Therefore $\T_{w_\bullet}$ commutes with $\omega'$. Since the graph automorphism $\tilde{\tau}$ commutes with $\omega'$ and $\T_{w_\bullet}$, we have $\theta'^2=\T_{w_\bullet}^2=\Xi((-1)^{\langle 2\rho_\bullet^\vee,\cdot\rangle})$, where $(-1)^{\langle 2\rho_\bullet^\vee,\cdot\rangle}(i)=(-1)^{\langle 2\rho_\bullet^\vee,\alpha_i\rangle}$, for $i\in\I$. The last equality follows from \cite[Proposition 2.2(2)]{Ko14}.

For any $i\in\I$ and $\lambda\in X$, we have $$\theta^2(E_i\one_\lambda)=\epsilon(i+\tau i)\theta'^2(E_i\one_\lambda)=\overline{\varsigma_i}\cdot\overline{\varsigma_{\tau i}}(-1)^{\langle 2\rho_\bullet^\vee,\alpha_i\rangle}E_i\one_\lambda=E_i\one_\lambda.$$ Similarly we have $\theta^2(F_i\one_\lambda)=F_i\one_\lambda$. Hence we proved that $\theta^2=id$.

To check $\theta$ restricts to $G_A$, it suffices to show that $\theta$ is compatible with the coproduct, that is $\theta\otimes\theta\circ \widehat{\Delta}=\widehat{\Delta}\circ\theta$, and $\theta(\one_0)=\one_0$. The second condition is trivial. For the first one, it is easy to verify that maps $\tilde{\tau}$, $\omega$ and $\Xi(\epsilon)$ are compatible with coproduct. For the automorphism $\T_{w_\bullet}$, it will suffice to check that $\WD(\T_i(u))=\T_i\otimes\T_i\circ\WD(u)$, for any $u\in {_A\widehat{\U}}$, which follows from $\T_i(u)=\overline{s_i}u\overline{s_i}^{-1}$, and $\WD(\overline{s_i})=\overline{s_i}\otimes \overline{s_i}$ (\cite[Lemma 2.2(e), 2.3 (c)]{Lu09}). Therefore we showed that $\theta(G_A)=G_A$.

Finally, the functoriality is clear by definition. We complete the proof.
\end{proof}

\begin{corollary}
We have a canonical bijection:
\[
\{\text{iso. classes of symmetric pairs $(G_k,\theta_k)$}\}\leftrightarrow\{\text{iso. classes of $\imath$root data}\}.
\]    
\end{corollary}

\begin{proof}
    Given a symmetric pair $(G_k,\theta_k)$, it corresponds to an $\imath$root datum as explained in Section \ref{sec:lkb}. On the other hand, given an $\imath$root datum, we have constructed a symmetric pair $(G_k,\theta_k)$ in Theorem \ref{prop:ainv}. It is clear that these two maps are inverse to each other (up to isomorphism). 
\end{proof}

\subsection{Points over integral domains}

We are now ready to describe points of the group scheme $\mathbf{G}^\imath$ over integral domains.

Let $A$ be an integral domain where $1+1\neq 0$ in $A$. Set $K_A=G_A^{\theta_A}$ be the subgroup of $G_A$ consisting of $\theta_A$ fixed points. It is called the \emph{symmetric subgroup} of $G_A$ corresponding to the involution $\theta_A$.  

Let $k$ be the algebraic closure of the quotient field of $A$. Then the characteristic of $k$ in not 2. By Section \ref{sec:regpq} $G_k$ is a connected reductive group over $k$. We take the maximal torus $T_k$ and the Borel subgroup $B_k$ as in Section \ref{sec:regpq}. Recall the involution $\theta_k:G_k\rightarrow G_k$ in the previous section. It follows from the definition that $\theta_k$ preserves $T_k$. We denote by $K_k^\circ$ the identity component of $K_k$. Then by \cite[Proposition~7]{Vu74} (cf. \cite[\S2.9]{Ri82}), we have \begin{equation}\label{eq:TK}
 K_k = T_k^{\theta_k} K^\circ_k.
\end{equation}

Recall the algebra $\RO_k^\imath$ and the subgroup $G_k^\imath\subset G_k$ from Section \ref{sec:def}.

\begin{theorem}\label{thm:Oik}
As closed subgroups of $G_k$, we have $G_k^\imath=K_k$. In particular, the coordinate ring of the symmetric subgroup $K_k$ is isomorphic to $\RO_k^\imath$.
\end{theorem}

\begin{proof}
We firstly show that $G_k^\imath\subset K_k$. It suffices to show the subalgebra $_k\dot{\U}^\imath \subset {}_k\widehat{\U}$ is fixed by $\theta_k$ pointwise. Since $_k\dot{\U}^\imath\cong k\otimes_\BZ {}_\BZ\dot{\U}^\imath$, and $_\BZ\dot{\U}^\imath$ is a subalgebra of $_\BC\dot{\U}^\imath$, we may assume that $k=\BC$. Since $_\BC\dot{\U}^\imath$ is generated by $\one_\zeta$ ($\zeta\in X_\imath$), $B_{i}\one_\zeta$ ($i\in\I$, $\zeta\in X_\imath$), $E_i\one_\zeta$ ($i\in\I_\bullet,\zeta\in X_\imath$) , we only need to check $\theta_\BC$ fixes these generators. This is direct, and we omit the computation.

Next we compare the Lie algebras of $G_k^\imath$ and $K_k$. Let $\mathfrak{g}$ be the Lie algebra of $G$. Recall from Section \ref{sec:regpq} that we have a canonical embedding $\mathfrak{g}\hookrightarrow{}_k\widehat{\U}$. Also recall that $\mathfrak{g}$ has generators (as a Lie algebra)
\[
e_i=\sum_{\lambda\in X}E_i\one_\lambda,\qquad f_i=\sum_{\lambda\in X}F_i\one_\lambda,\qquad h_\mu=\sum_{\lambda\in X}\langle \mu,\lambda\rangle\one_\lambda,
\]
for any $i\in\I$, and $\mu\in Y$. Here we are abusing notations by using the same notations to denote the elements after base change.

The involution $\theta_k$ on $_k\widehat{\U}$ also restricts to a (Lie algebra) involution on $\mathfrak{g}$, which is the same as the differentiation of the involution on the group $G_k$. We still write $\theta_k$ to denote the involution on $\mathfrak{g}$.


Let $\mathfrak{k}$ be the Lie algebra of $K_k$. By \cite[\S9.4]{Bo91}, we have $\mathfrak{k}=\mathfrak{g}^{\theta_k}$. By \cite[Lemma 2.8]{Ko14} the subalgebra $\mathfrak{k}$ is generated by $f_i+\theta_k(f_{ i})$ ($i\in\I_\circ$), $e_j$, $f_j$ ($j\in \I_\bullet$) and $h_\mu$ ($\mu\in Y^\imath$).

We write $\mathfrak{g}^\imath\subset \mathfrak{g}$ to be the Lie algebra of $G^\imath_k$. Then it follows from the definition that $\mathfrak{g}^\imath=\mathfrak{g}\cap{_k\widehat{\U}^\imath}$. For any $i\in\I_\circ$, $j\in\I_\bullet$, $\mu\in Y^\imath$, we have
\[
f_i+\theta_k(f_{ i})=\sum_{\zeta\in X_\imath}B_{i}\one_\zeta,\quad f_j=\sum_{\zeta\in X_\imath}F_i\one_\zeta, \quad e_j=\sum_{\zeta\in X_\imath}E_i\one_\zeta,\quad  h_\mu=\sum_{\zeta\in X_\imath}\langle \mu,\zeta\rangle \one_\zeta.
\]
In particular we deduce that $\mathfrak{k}\subset\mathfrak{g}^\imath$. Combining with $G_k^\imath\subset K_k$, we deduce that $K_k^\circ$ is contained in $G_k^\imath$. 

Thanks to (\ref{eq:TK}), it remains to show that $T_k^{\theta_k}$ is contained in $G_k^\imath$.
Take any $\xi=\sum_{\lambda\in X}n_\lambda\one_\lambda$ in $T_k^{\theta_k}$. Then $n_\lambda\in k^\times$  and $n_{\lambda'}n_{\lambda''}=n_{\lambda'+\lambda''}$, for any $\lambda',\lambda''\in X$. Moreover we have $n_\mu=n_{\theta\mu}$, for any $\mu\in X$. Therefore $n_{\mu-\theta\mu}=1$. Hence $n_{\lambda'}=n_{\lambda''}$, whenever $\overline{\lambda'}=\overline{\lambda''}$ in $X_\imath$. For any $\zeta\in X_\imath$, we define $n_\zeta=n_\lambda$, for any $\lambda\in X$ with $\overline{\lambda}=\zeta$. Then we can write $\xi=\sum_{\zeta\in X_\imath}n_\zeta\one_\zeta \in {}_k\widehat{\U}^\imath$. 

We finish the proof now.
\end{proof}

\begin{corollary}
As subgroups of $G_A$, we have $G_A^\imath=K_A$.
\end{corollary}

\begin{proof}
The proof for $G_A^\imath\subset K_A$ is the same as the first part of the proof of Theorem \ref{thm:Oik}. For the other side inclusion, we embed all the objects into the ambient space $_k\widehat{\U}$. It follows from definitions that 
\[
K_A\subset {_A\widehat{\U}}\cap K_k={_A\widehat{\U}}\cap G_k^\imath={_A\widehat{\U}}\cap{_k\widehat{\U}^\imath}\cap G_k={_A\widehat{\U}^\imath}\cap G_A=G_A^\imath.
\]
We completed the proof.
\end{proof}

\subsection{Quantization}\label{sec:qt}

In this section we apply the construction of $\RO_A$ and $\RO_A^\imath$ to $A=\A$. By \cite[\S 3.1]{Lu09}, the $\A$-algebra $\RO_\A$ naturally carries a structure of a Hopf algebra, which is a quantization of the \emph{standard Poisson-Lie group structure} on $G_\mathbb{C}$. It is shown in \cite[Proposition 3.1]{Song24} that $K_{\mathbb{C}}\subset G_{\mathbb{C}}$ is a \emph{coisotropic subgroup}. We refer the readers to \cite[Chapter~1]{CP95} for the precise definitions regarding Poisson-Lie groups.

We will focus on the quantum counterparts of these notions, and show that $\RO_\A^\imath$ is a quantization of $K_\mathbb{C}$ in the sense of Ciccoli \cite{Cic97}.

\begin{definition}(\cite[Definition 4.3]{Cic97})
    A coisotropic quantum right subgroup $C$ of $\RO_\A$ is an $\A$-coalgebra, such that

    (1) $C$ is a $\RO_\A$-right module;

    (2) there exists a surjective linear map $r:\RO_\A\rightarrow C$, which is a morphism of $\RO_\A$-modules and of coalgebras.
\end{definition}

\begin{thm}
    $\RO_\A^\imath$ is a coisotropic quantum right subgroup of $\RO_\A$.
\end{thm}

\begin{proof}
    Since $_\A\dot{\U}^\imath$ is an associative $\A$-algebra, the $\A$-linear space $\RO_\A^\imath$ is naturally a coalgebra by the definition. The coideal structure $\Delta:\U^\imath\rightarrow \U^\imath\otimes\U$ together with finiteness conditions induce a right $\RO_\A$-action on $\RO_\A^\imath$. The surjective map $r:\RO_\A\rightarrow\RO_\A^\imath$ follows from Section \ref{sec:def}. It is a morphism of $\RO_\A$-modules since the coideal structure $\Delta:\U^\imath\rightarrow \U^\imath\otimes\U$ is the restriction of the coproduct on $\U$. It is a morphism of coalgebras since $\U^\imath$ is a subalgebra of $\U$. 
\end{proof}

\begin{remarks}
    The subspace $\RO_\A^\imath$ is not a Hopf algebra in general. This reflects the fact that $K_\mathbb{C}$ is not a Poisson subgroup of $G_{\mathbb{C}}$ in general, that is, the embedding $K_\mathbb{C}\rightarrow G_\mathbb{C}$ is not Poisson. It is known that the symmetric subgroup $K_{\mathbb{C}}$ is reductive. Therefore one can apply Lusztig's construction to obtain a Hopf algebra $\RO_{K,\A}$. However, since in general $K_\mathbb{C}$ is not a Poisson subgroup one cannot expect a quotient map $\RO_\A\rightarrow\RO_{K,\A}$ on the quantum level. Hence in order to construct the symmetric subgroup schemes, the consideration of quantum symmetric pairs is essential.
\end{remarks}


\begin{thebibliography}{99}

\bibitem{Ar62}
S. Araki.
\textit{On root systems and an infinitesimal classification of irreducible symmetric spaces}.
J. Math. Osaka City Univ., 13:1--34, 1962.

\bibitem{BW18}
H. Bao and W. Wang.
\textit{Canonical bases arising from quantum symmetric pairs}.
Invent. Math., 213(3):1099--1177, 2018.
doi:10.1007/s00222-018-0801-5.

\bibitem{BLV}
M. Brion, D. Luna, and T. Vust.
\textit{Espaces homogènes sphériques}.
Invent. Math., 84:617--632, 1986.
doi:10.1007/BF01388749.

\bibitem{Bo91}
A. Borel.
\textit{Linear Algebraic Groups}.
Graduate Texts in Mathematics 126.
Springer-Verlag, New York, 2nd edition, 1991.

\bibitem{BS22}
H. Bao and J. Song.
\textit{Symmetric Subgroup Schemes, Frobenius Splittings, and Quantum Symmetric Pairs}.
arXiv:2212.13426 [math.RT], 2022.
\url{https://arxiv.org/abs/2212.13426}.

\bibitem{BS24}
H. Bao and J. Song.
\textit{Coordinate rings on symmetric spaces}.
arXiv:2402.08258 [math.RT], 2024.
\url{https://arxiv.org/abs/2402.08258}.

\bibitem{BS25}
H. Bao and J. Song.
\textit{Dual canonical bases and embeddings of symmetric spaces}.
arXiv:2505.01173 [math.RT], 2025.
\url{https://arxiv.org/abs/2505.01173v1}.

\bibitem{BZSV}
D. Ben-Zvi, Y. Sakellaridis, and A. Venkatesh.
\textit{Relative Langlands Duality}.
arXiv:2409.04677 [math.RT], 2024.
\url{https://arxiv.org/abs/2409.04677}.

\bibitem{Ch55}
C. Chevalley.
\textit{Sur Certains Groupes Simples}.
Tohoku Math. J. (2), 7:14--66, 1955.
doi:10.2748/tmj/1178245104.

\bibitem{Ch61}
C. Chevalley.
\textit{Certains Schémas de Groupes Semi-simples}.
In \textit{Séminaire Bourbaki, Vol. 6}, pages 219--234.
Soc. Math. France, Paris, 1961.

\bibitem{Cic97}
N. Ciccoli.
\textit{Quantization of Co-Isotropic Subgroups}.
Lett. Math. Phys., 42:123--138, 1997.

\bibitem{CP95}
V. Chari and A. N. Pressley.
\textit{A Guide to Quantum Groups}.
Cambridge University Press, Cambridge, 1995.

\bibitem{Jan03}
J. C. Jantzen.
\textit{Representations of Algebraic Groups}.
Mathematical Surveys and Monographs 107.
American Mathematical Society, Providence, RI, 2nd edition, 2003.

\bibitem{Ko66}
B. Kostant.
\textit{Groups over $\mathbb{Z}$}.
In \textit{Algebraic Groups and Discontinuous Subgroups}, pages 90--98.
Amer. Math. Soc., Providence, R.I., 1966.

\bibitem{Ko14}
S. Kolb.
\textit{Quantum Symmetric Kac-Moody Pairs}.
Adv. Math., 267:395--469, 2014.
doi:10.1016/j.aim.2014.08.010.

\bibitem{Le99}
G. Letzter.
\textit{Symmetric Pairs for Quantized Enveloping Algebras}.
J. Algebra, 220(2):729--767, 1999.
doi:10.1006/jabr.1999.8015.

\bibitem{Lu93}
G. Lusztig.
\textit{Introduction to Quantum Groups}.
Modern Birkhäuser Classics.
Birkhäuser, Boston, MA, 1993.
doi:10.1007/978-0-8176-4717-9.

\bibitem{Lu09}
G. Lusztig.
\textit{Study of a $\mathbb{Z}$-form of the Coordinate Ring of a Reductive Group}.
J. Amer. Math. Soc., 22(3):739--769, 2009.
doi:10.1090/S0894-0347-08-00603-6.

\bibitem{Lu23}
G. Lusztig.
\textit{The Quantum Group $\dot{U}$ and Flag Manifolds over the Semifield $\mathbb{Z}$}.
Bull. Inst. Math. Acad. Sinica (N.S.), 18(3):235--267, 2023.
doi:10.21915/BIMAS.2023301.

\bibitem{Ri82}
R. W. Richardson.
\textit{Orbits, invariants, and representations associated to involutions of reductive groups}.
Invent. Math., 66(2):287--312, 1982.
doi:10.1007/BF01389396.

\bibitem{SV}
Y. Sakellaridis and A. Venkatesh.
\textit{Periods and Harmonic Analysis on Spherical Varieties}.
Astérisque 396.
Soc. Math. France, Paris, 2017.

\bibitem{Song24}
J. Song.
\textit{Quantum duality principle and quantum symmetric pairs}.
arXiv:2403.05167 [math.QA], 2024.
\url{https://arxiv.org/abs/2403.05167}.

\bibitem{Spr87}
T. A. Springer.
\textit{The Classification of Involutions of Simple Algebraic Groups}.
J. Fac. Sci. Univ. Tokyo Sect. IA Math., 34(3):655--670, 1987.

\bibitem{Ste68}
R. Steinberg.
\textit{Endomorphisms of Linear Algebraic Groups}.
Mem. Amer. Math. Soc. 80.
Amer. Math. Soc., Providence, R.I., 1968.

\bibitem{Vu74}
T. Vust.
\textit{Opération de groupes réductifs dans un type de cônes presque homogènes}.
Bull. Soc. Math. France, 102:317--333, 1974.

\bibitem{Wa23}
H. Watanabe.
\textit{Stability of $\imath$-canonical bases of locally finite type}.
Transformation Groups, 2024. 

\bibitem{Wa23a}
H. Watanabe.
\textit{Stability of ıCanonical Bases of Irreducible Finite Type of Real Rank One}.
Represent. Theory, 27:1--29, 2023.
doi:10.1090/ert/639.

\bibitem{Wa24b}
H. Watanabe.
\textit{Integrable modules over quantum symmetric pair coideal subalgebras}.
 Journal of the Institute of Mathematics of Jussieu. 2025;24(6):2257-2282. doi:10.1017/S1474748025100984

\bibitem{Wang22}
W. Wang.
\textit{Quantum Symmetric Pairs}.
In \textit{Proc. Int. Congr. Math. 2022}, volume 4, pages 3080--3102.
EMS Press, 2022.
doi:10.4171/ICM2022/76.

\end{thebibliography}
\end{document}